\documentclass[11pt,reqno,fleqn]{amsart}

\usepackage{amsfonts,amsmath,amsthm,amssymb, enumerate}
\usepackage{pdfsync}

\newtheorem{theorem}{Theorem}[section]
\newtheorem{thm}{Theorem}[section]

\newtheorem{prop}[theorem]{Proposition}

\newtheorem{lemma}[theorem]{Lemma}

\theoremstyle{definition}

\numberwithin{equation}{section}

\newcommand{\R}{\ensuremath{\mathbb{R}}}
\newcommand{\Z}{\ensuremath{\mathbb{Z}}}

\newcommand{\N}{\ensuremath{\mathbb{N}}}

\newcommand{\MM}{\ensuremath{\mathcal{M}}}
\newcommand{\NN}{\ensuremath{\mathcal{N}}}

\newcommand{\1}{\ensuremath{\mathbf{1}}}

\newcommand{\ls}{\ensuremath{\lesssim}}
\newcommand{\gs}{\ensuremath{\gtrsim}}
\newcommand{\tr}{\ensuremath{\triangle}}
\newcommand{\subs}{\ensuremath{\subseteq}}

\newcommand{\eps}{\ensuremath{\varepsilon}}
\newcommand{\la}{\ensuremath{\lambda}}

\newcommand{\al}{\ensuremath{\alpha}}
\newcommand{\be}{\ensuremath{\beta}}
\newcommand{\ga}{\ensuremath{\gamma}}

\newcommand{\de}{\ensuremath{\delta}}

\newcommand{\si}{\ensuremath{\sigma}}

\newcommand{\om}{\ensuremath{\omega}}

\newcommand{\eq}{\begin{equation}
\newcommand{\ee}{\end{equation}}}

\numberwithin{equation}{section}

\begin{document}
\title{A Roth-type theorem for dense subsets of $\R^d$}
\author{Brian Cook\quad\'Akos Magyar\quad Malabika Pramanik}
\thanks{2010 Mathematics Subject Classification. 05D10, 42B20.\\
The first author is supported by NSF grant DMS1147523. The second and third authors are respectively supported by grant ERC-AdG. 321104
and an NSERC Discovery grant.}
\begin{abstract} Let $1 < p < \infty$, $p\neq 2$. We prove that if
  $d\geq d_p$ is sufficiently large, and $A\subs\R^d$ is a measurable
  set of positive upper density then there exists $\la_0=\la_0(A)$
  such for all $\la\geq\la_0$ there are $x,y\in\R^d$ such that
  $\{x,x+y,x+2y\}\subs A$ and $|y|_p=\la$, where $||y||_p=(\sum_i
  |y_i|^p)^{1/p}$ is the $l^p(\mathbb R^d)$-norm of a point
  $y=(y_1,\ldots,y_d)\in\R^d$. This means that dense subsets of $\R^d$
  contain 3-term progressions of all sufficiently large gaps when the
  gap size is measured in the $l^p$-metric. This statement is known to
  be false in the Euclidean $l^2$-metric as well as in the $l^1$ and
  $\ell^{\infty}$-metrics. One of the goals of this note is to
  understand this phenomenon. A distinctive feature of the proof is
  the use of multilinear singular integral operators, widely studied
  in classical time-frequency analysis, in the estimation of forms
  counting configurations. 
\end{abstract}
\maketitle


\section{Introduction.} A main objective of Ramsey theory is
the study of geometric configurations in large but otherwise arbitrary
sets. A typical problem in this area reads as follows: given a set
$S$, a family $\mathcal F$ of subsets of $S$ and a positive integer
$r$, is it true that any $r$-colouring of $S$ yields some
monochromatic configuration from $\mathcal F$? More precisely, for any
partition of $S = S_1 \cup \cdots \cup S_r$ into $r$ subsets, does
there exist $i \in \{1, 2, \cdots, r\}$ and $F \in \mathcal F$ such
that $F \subseteq S_i$? In discrete (respectively Euclidean) Ramsey
theory $S$ is generally $\mathbb Z^d$ (respectively $\mathbb R^d$),
and sets in $\mathcal F$ are geometric in nature. For example, if $X$
is a fixed finite subset of $\mathbb R^d$, such as a collection of equally spaced
collinear points or vertices of an isosceles right triangle, then $\mathcal F =
\mathcal F(X)$ could be the collection of all isometric copies or all
homothetic copies of $X$ in $S$. A colouring theorem refers to a
choice of $S$ and $\mathcal F$ for which the answer to the above-mentioned question is
yes. Such theorems are often
consequences of sharper, more quantitative statements known as density
theorems. A fundamental result with $S = \mathbb N = \{1, 2, \cdots
\}$ is Szemer\'edi's theorem \cite{Sz}, which states that if $E \subseteq
\mathbb N$ has positive upper density, i.e., \[ \limsup_{N \rightarrow
\infty} \frac{|E \cap \{1, \cdots, N\}|}{N} > 0, \] then $E$ contains
a $k$-term arithmetic progression for every $k$. This in particular
implies van der Waerden's theorem \cite{{vdW},{GRS}}, which asserts that given $r
\geq 1$, any $r$-colouring of $\mathbb N$ must produce a $k$-term
monochromatic progression, i.e., a homothetic copy of $\{1, 2, \cdots,
k\}$.
\vskip0.2in
In this note, we will be concerned with strong density theorems in $S
= \mathbb R^d$, with similar implications to colouring problems. A
basic and representative result in the field
states that for any $d \geq 2$, a set $A\subs\R^d$ of
positive upper Banach density contains all large distances. i.e., for
every sufficiently large $\la\geq\la_0(A)$ there are points $x,x+y\in
A$ such that $||y||_2=\la$. Recall that the positive upper Banach
density of $A$ is defined as
\[\bar{\de}(A):=\limsup_{N\to\infty}\sup_{x\in\R^d} \frac{|A\cap (x+[0,N]^d)|}{N^d}.\]
Replacing $\limsup$ in the above definition with $\liminf$ leads to
the definition of the lower Banach density
$\underline{\delta}(A)$. The result quoted above was obtained independently, along with various
generalizatons, by a number of authors, for example Furstenberg,
Katznelson and Weiss \cite{FKW}, Falconer and Marstrand
\cite{FM}, and Bourgain \cite{Bo}. In the notation of the previous
paragraph, this result addresses the situation
where $X$ is a two-point set and $\mathcal F(X)$ is enlarged to allow for
expansions as well. The conclusion ensures existence of configurations in a very
strong sense, namely for {\em{all sufficiently large}} dilates of $X$ instead
of merely {\em{some}}.
\vskip0.2in
A natural question to ask is whether similar statements exist that involve a
configuration with a greater number of points. If one looks for some large dilate of a given
configuration, such results are well-known in the
discrete regime of the integer lattice, under suitable assumptions of largeness on the
underlying set. These results can often be translated to existence
of configurations in the Euclidean setting as well. For instance,
Roth's theorem \cite{Ro} in the integers states that a subset of $\Z$
of positive upper density contains a three-term arithmetic progression
$\{x,x+y,x+2y\}$ and it easily implies that a measurable set $A\subs\R$ of
positive upper density contains a three-term progression whose gap size can
be arbitrarily large. Results ensuring all sufficiently large dilates
of a configuration in a set of positive Banach density are typically
more difficult. Bourgain \cite{Bo} shows that if $X$ is any non-degenerate
$k$-point simplex in $\mathbb R^d$, $d \geq k \geq 2$ (i.e. $X$ spans
a $(k-1)$-dimensional space), then any subset of $\mathbb R^d$ of
positive upper Banach density contains a congruent copy of $\lambda X$
for all sufficiently large $\lambda$.
\vskip0.2in
On the other hand, a simple example given in
\cite{Bo} shows that there is a set $A\subs\R^d$ in any dimension $d
\geq 1$, such that the gap lengths of all 3-progressions in $A$ do
not contain all sufficiently large numbers. In other words, the
result of \cite{Bo} is false for the degenerate
configuration $X = \{0, e_1, 2e_1 \}$, where $e_1$ is the canonical
unit vector in the $x_1$-direction. More precisely, the
counterexample provided in \cite{Bo} is the set $A$ of points
$x\in\R^d$ such that $|||x||_2^2-m|\leq\frac{1}{10}$ for some $m\in
\N$. The parallelogram identity $2||y||_2^2 = ||x||_2^2 + ||x +
2y||_2^2 - 2 ||x+y||_2^2$
then dictates that $\bigl|||y||_2^2- \frac{\ell}{2} \bigr|\leq \frac{4}{10}$ (for
some $\ell\in\N$) for any progression $\{x,x+y,x+2y\}\subs A$. Thus
the squares of the gap lengths are restricted to lie close to the
half-integers, and therefore cannot realize all sufficiently large
numbers.
\vskip0.2in
The counterexample above has an interesting connection with a result
in Euclidean Ramsey theory due to Erd\"os et al \cite{EGMRSS}. Let us recall \cite{Gr}
that a finite point set $X$ is said to be {\em{Ramsey}} if for every
$r \geq 1$, there exists $d = d(X,r)$ such that any $r$-colouring of
$\mathbb R^d$ contains a congruent copy of $X$. A result in \cite{EGMRSS}
states that every Ramsey configuration $X$ is spherical, i.e., the
points in $X$ lie on an Euclidean sphere. (The converse statement is
currently an open conjecture due to Graham \cite{Gr}). Since a set of three
collinear points is non-spherical, it is natural to ask whether
Bourgain-type counterexamples exist for any non-spherical $X$. This
question was posed by Furstenberg and answered in the affirmative by Graham
\cite{Gr}. We state his result below for convenience.
\begin{thm}[Graham \cite{Gr}] \label{graham}
Let $X$ be a finite non-spherical set. Then for any $d \geq 2$, there
exists a set $A \subseteq \mathbb R^d$ with $\bar{\delta}(A) > 0$ and
a set $\Lambda \subset \mathbb R$ with $\underline{\delta}(\Lambda) >
0$ so that $A$ contains no congruent copy of $\lambda X$ for any
$\lambda \in \Lambda$.
\end{thm}
It is interesting to observe that while Bourgain's counterexample
prevents an existence theorem for three term arithmetic progressions
of all sufficiently large {\em{Euclidean}} gaps, it does not exclude the validity of such a result when
the gaps are measured using some other metric on $\R^d$ that does
not obey the parallelogram law. The aim of this note is to prove that such results
do indeed exist  for the $l^p$ metrics $||y||_p:=(\sum_{i=1}^d
|y_i|^p)^{1/p}$ for all $1 < p < \infty$, $p\neq 2$. In this
sense, a counterexample as described above is more the exception rather
than the rule.
\vskip0.2in
Variations of our arguments also work for other
metrics given by specific classes of positive homogeneous polynomials of degree at least
4 and those generated by symmetric convex bodies with special
structure. Results of the first type were obtained in the finite field setting by
the first two authors \cite{CM}. Also, the arguments here can be
applied to obtain similar results for certain other degenerate point
configurations. We hope to pursue these extensions elsewhere.

\section{Main results}

\begin{thm}\label{thm2.1} Let $1 < p < \infty$, $p\neq 2$. Then there
  exists a constant $d_p \geq 2$ such that for $d\geq d_p$ the following
  holds. Any measurable set $A\subs \R^d$ of positive upper Banach
  density contains a three-term arithmetic progression  $\{x,x+y,x+2y\}\subs A$ with gap $||y||_p=\la$ for all sufficiently large $\la\geq \la(A)$.
\end{thm}
\vskip0.2in
{\em{Remarks:}}
\begin{enumerate}[(a)]
\item The result is sharp in the range of $p$. Easy variants of the example in \cite{Bo} show
that Theorem \ref{thm2.1} and in fact even the two-point results of \cite{{FKW}, {Bo},
  {FM}} cannot be true for $p = 1$ and $p =
\infty$. Indeed, if $A = \mathbb Z^d + \epsilon_0 [-1, 1]^d$ for some
small $\epsilon_0 > 0$, then
on one hand $A$ is of positive upper Banach density. On the other hand, if
$x, x+y \in A$ for some $y \ne 0$, then both $||y||_{\infty}$ and $||y||_1$ are
restricted to lie within distance $O(\epsilon_0)$ from some positive integer.
\vskip0.1in
\item Indeed, counterexamples similar to \cite{Bo} and the above can be constructed
  for norms given by a symmetric, convex body, a nontrivial part of whose
  boundary is either flat or coincides with an $l^2$-sphere. An appropriate
  formulation of a positive result for a general norm, and indeed the
  measuring failure of the parallelogram law for such norms, remains
  an interesting open question.   
\vskip0.1in   
\item We do not know whether the $p$-dependence of the dimensional threshold $d_p$
stated in the theorem is an artifact of our proof. In our analysis,
$d_p$ grows without bound as $p
\nearrow \infty$, while other implicit constants involved in the proof
blow up near $p = 1$ and $p = 2$. See in particular Proposition \ref{prop2.2}
and Lemma
\ref{lem4.2}. It would be of interest to determine whether Theorem
\ref{thm2.1} holds for all $d \geq 2$ for the specified values of
$p$. It is worth pointing out that certain colouring theorems work
only in high enough dimensions, but we are currently  unaware of any lower
dimensional phenomena that might preclude an analogue of Theorem
\ref{thm2.1} for small $d$.
\vskip0.1in 
\item Since three collinear points do not lie on an $l^p$-sphere for
  any $p \in (1, \infty)$, Theorem \ref{thm2.1} shows that a result of the type Theorem
  \ref{graham} is in general false for an $l^p$-sphere if $p \ne
 1, 2, \infty$. Thus any connection between Ramsey-like properties and the notion
  of sphericality appears to be a purely $l^2$ phenomenon.
\end{enumerate}

\subsection{Overview of proof} We describe below the main elements of
the proof. Details will be provided in subsequent sections.
\vskip0.2in
Our main observation is a stronger finitary version of Theorem
\ref{thm2.1} for bounded measurable sets.
\begin{thm}\label{thm2.2} Let $1 < p < \infty$, $p\neq 2$. Let $d\geq
  d_p$, $\de>0$ and let $N\geq N(\de)$ be sufficiently large. Then for
  any measurable set $A\subs [0,N]^d$ of measure $|A|\geq \de\,N^d$
  the following holds.

For any lacunary sequence $1<\la_1<\ldots <\la_J\ll N$ with
$\la_{j+1}\geq 2\la_j$ and $J\geq J(\de)$, there exists a three term
arithmetic progression $\{x,x+y,x+2y\}\subs A$ such that $||y||_p=\la_j$ for some $1\leq j\leq J$,
\end{thm}

\begin{proof}[Proof of Theorem \ref{thm2.1}] Theorem \ref{thm2.2} implies Theorem \ref{thm2.1}. Indeed, assume that
Theorem \ref{thm2.1} does not hold. Then there exists an infinite
sequence $\{\la_j\}_{j=1}^\infty\subs\N$ such that $\la_j\neq ||y||_p$ for
any $j$ and any $y$ which is the gap of a 3-progression contained in
$A$. Without loss of generality the sequence may be assumed to be
lacunary, i.e. $\lambda_{j+1} \geq 2 \lambda_j$ for all $j$. After setting
$\delta = \overline{\delta(A)}/2$, fix any sufficiently large
$J=J(\de)$ and any sufficiently large box $B_N$ of size
$N=N(\de,\la_J)$ on which the density of $A$ is $|A|/N^d\geq\de$. By
translation invariance we may assume $B_N=[0,N]^d$, contradicting
Theorem \ref{thm2.2}.
\end{proof}
\vskip0.2in

\subsubsection{A counting function and its variants} For the rest of the paper we fix a finite exponent $p>1$, $p\neq 2$, and
for simplicity of notation write $|y| = ||y||_p$.
We start by counting three term arithmetic progressions
$P=\{x,x+y,x+2y\}$ contained in $A$ via a positive measure $\si_\la$
supported on the $l^p$-sphere $S_\la=\{y\in\R^d;\
|y|=\la\}$. Let $f:=\1_A$ be the indicator function of a measurable set $A\subs [0,N]^d$.
As is standard in enumerating configurations, we introduce the
counting function
\eq\label{2.1}
\mathcal{N}_\la(f):=\int_{\R^d}\int_{S_\la} f(x)f(x+y)f(x+2y)\,d\si_\la(y)\,dx.
\ee
Clearly if $\NN_\la(f)>0$ then $A$ must contain a 3-progression
$x,x+y,x+2y$ with $|y|=\la$. We will define the measure $\si_\la$ via the oscillatory integral
\begin{equation} \label{2.2}
\si_\la(y):=\,\la^{-d+p}\int_\R e^{i\,(|y|^p-\la^p)t}\,dt.
\end{equation}
It is well-known (see \cite{Ho}, Ch.2) that the above oscillatory integral
defines an absolutely continuous measure with respect to the surface
area measure on $S_\la$ whose density function is $|\nabla Q(y)|^{-1}$
with $Q(y)=|y|^p$. The normalizing factor $\la^{-d+p}$ is inserted to
ensure that $\si_\la (S_\la)=\si_1(S_1)>0$, which is independent of $\la$.\\

Let $\psi$ be a Schwarz function such that $0\leq\psi\leq 1$,
$\psi(0)=1$ and $\widehat{\psi}\geq 0$ is compactly supported. Define the quantity
\eq\label{2.3}
\om_\la(y):=\la^{-d+p}  \int_\R e^{it\,(|y|^p-\la^p)} \psi(\la^pt)\,dt.
\ee
Note that by scaling
\begin{equation} \label{omega-scaling}
  \om_\la(y)=\la^{-d}\om(y/\la)=\la^{-d}\,\widehat{\psi}\,(\,|y/\la|^p-1\,),\end{equation}
hence $\omega_{\lambda}$ is compactly supported on $B(0;C\lambda)$
with
\[\int \om_\la(y)\,dy=\int \om(y)\,dy=C_\om>0.\]
When the subscript is omitted it should be assumed that $\la=1$. Also define
\eq\label{2.4}
\MM_\la(f):=\int_{\R^d}\int_{\R^d} f(x)f(x+y)f(x+2y)\,\om_\la(y)\,dy\,dx.
\ee
The first step is to show that this quantity is large.
\begin{prop}\label{prop2.1} Let $0<\de\leq 1$ and let $A \subseteq
  [0,N]^d$ be such that $|A| \geq \delta N^d$.
Then there exists a constant $c(\de)>0$ depending only on $\de$ such that for $0<\la\ll N$,
\eq\label{2.5} \MM_\la(\mathbf 1_A)\,\geq\,c(\de)\,N^d.\ee
\end{prop}
As we will see in Section \ref{main term section}, the proof of this
proposition is in essence
Roth's theorem adapted to the Euclidean setting (see \cite{Bo}).
\vskip0.2in

Next, we define a variant of $\mathcal M_{\lambda}$ indexed by a small $\eps>0$ which is a good approximation to $\NN_\la(f)$. Let
\eq\label{2.6}
\om^\eps_\la (y):=\la^{-d+p}\int_\R e^{it\,(|y|^p-\la^p)} \psi(\eps\la^p t)\,dt.
\ee
It is easy to see that
\eq\label{2.7}
\om^\eps_\la(y)=\la^{-d}\eps^{-1}\widehat{\psi}\left(\frac{|y/\la|^p-1}{\eps}\right)=\la^{-d}\om^\eps(y/\la).
\ee
Define
\eq\label{2.8}
\MM^\eps_\la(f):=\int_{\R^d}\int_{\R^d} f(x)f(x+y)f(x+2y)\,\om^\eps_\la(y)\,dx\,dy.
\ee
We establish the error estimate
\begin{prop}\label{prop2.2} Let $f:[0,N]^d\to [-1,1]$ and let $\
  0<\eps<1$. Then there exist constants $\ga=\ga_p>0$ and $C_{p,d} >
  0$ both independent of $\lambda$ such that for $0<\la\ll N$,
\eq\label{2.9}
|\NN_\la(f)-\MM^\eps_\la(f)|\leq C_{p,d} \eps^{d\ga_p-1}\,N^d.
\ee
In particular, $\gamma_p$ is independent of $d$. The constant $C_{p,d} \nearrow \infty$ as $p \rightarrow 1$ or $2$,
while $\gamma_p \rightarrow 0$ as $p \nearrow \infty$.
\end{prop}

The proof of Proposition \ref{prop2.2} is based on two facts that may
be of independent interest. The
first is an inequality showing that the so-called $U^3$-uniformity
norm of Gowers \cite{Gow} controls expressions like $\NN_\la(f)$. Let
us recall the definition of the $U^3$ norm for a compactly supported bounded measurable function $g$:
\eq\label{2.10}
\|g\|_{U^3(\R^d)}^8=\int_{(x,y)\in\R^d \times \mathbb R^{3d}} \Bigl(\prod_{\nu \in \{0,1\}^3} \bar{g}^\nu (x+\nu_1y_1+\nu_2y_2+\nu_3y_3)\Bigr)\,dxdy,
\ee
where $\nu_1,\nu_2,\nu_3$ can take the values $0$ or $1$,
$\bar{g}^\nu=\bar{g}$ if $\nu_1+\nu_2+\nu_3$ is odd and
$\bar{g}^\nu=g$ otherwise.
\vskip0.2in

\begin{lemma}\label{lem2.1} Let $f:[0,N]^d\to [-1,1]$ and let $\ 0<\eps<1$. Then for $0<\la\ll N$ one has
\eq\label{2.11}
|\NN_\la(f)-\MM^\eps_\la(f)|\,\ls\,N^d\,\|\si-\om^\eps\|_{U^3}.
\ee
\end{lemma}
While it is not apriori clear how to define the $U^3$-norm of the measure
$\si$ defined in \eqref{2.2}, we note that $\om^\eps\to\si$ weakly as
$\eps\to 0$. To prove \eqref{2.11}, we  first establish that
$\{\om^\eta\}$ is a Cauchy sequence
with respect to the $U^3$-norm and then define
$\|\si - \omega^{\epsilon}\|_{U^3}:=\lim_{\eta\to 0} \|
\omega^{\eta} - \om^\eps \|_{U^3}$. \\

We also show that
\begin{lemma}\label{lem2.2} Let $0<\eps<1,\ p>1,\,p\neq 2$. If $d>8r_p$ with $r_p=\max (p+1,2p-1)$, then one has
\eq\label{2.12}
\|\si-\om^\eps\|_{U^3} \leq C_{p,d}\eps^{\frac{d}{8r_p}-1}.
\ee
where the constant $C_{p,d}$ has the same behaviour as described in
Proposition \ref{prop2.2}.
\end{lemma}
Let us note in passing that \eqref{2.11} and \eqref{2.12} yield
\eqref{2.10} with $\gamma= 1/(8r_p)$. The proof of Lemma \ref{lem2.2} uses in an essential way that on $\R^d$ the norm is
defined by the expression $|y|^p=\sum_i |y_i|^p$ for some $p>1,\ p\neq
2$. In particular, a pivotal role is played by the fact that, for $p \ne 1,2$ and $x, x+y, x+2y
\in \mathbb R_{+}^d$, \[ |x|^p + |x + 2y|^p - 2|x + y|^p =
\sum_{i=1}^{d} \bigl[x_i^p + (x_i+2y_i)^p - 2(x_i +
y_i)^p \bigr] \]  does not vanish identically.
It is worth remarking that this part of the argument fails both for
the standard Euclidean norm and the $l^1$-norm. Indeed, estimate
\eqref{2.12} does not hold for either the $l^2$ or the $l^1$-metric.
\vskip0.2in

\subsubsection{Multilinear Calder\'on-Zygmund singular integral
  operators} The final ingredient in the proof of Theorem \ref{thm2.2}
is an estimate given in \cite{MTT} for certain multilinear operators similar
to the bilinear Hilbert transform. In order to describe the form in
which we need this estimate,
let us fix $0 < \epsilon \ll 1$ and define the constant $c_1(\epsilon)$ as follows,
\eq\label{2.12.5}
c_1(\eps)\int_{\R^d}\om (y)\,dy = \int_{\R^d}\om^\eps (y)\,dy.
\ee
We see below in Lemma \ref{lemma-c1e} that $c_1(\eps)\approx 1$,
i.e. is bounded by two positive constants depending only on the
dimension $d$.
Write $k^\eps(y):=\om^\eps_{\la}(y)-c_1(\eps)\,\om_{\la_j}(y)$,
and
\eq\label{2.14}
\mathcal{E}_{\la}(f):=\MM^\eps_{\la}(f)-c_1(\eps)\MM_{\la}(f)\,=\,\int_{\R^d}\int_{\R^d} f(x)f(x+y)f(x+2y)\,k^\eps(y)\,dy\,dx,
\ee
so that by \eqref{2.12.5} one has the cancellation property
\eq\label{2.14.5}
\int_{\R^d} k^\eps (y)\,dy =\int_{\R^d} (\om^\eps_{\la}(y)-c_1(\eps)\,\om_{\la}(y))\,dy = \int_{\R^d} (\om^\eps(y)-c_1(\eps)\,\om (y))\,dy\,=0.
\ee
The key estimate concerning the operator $\mathcal E_{\lambda}$ is the
following:
\begin{prop}\label{tf-prop}
Suppose that $\{\lambda_j: 1 \leq j \leq J \}$ is a lacunary sequence (finite or
infinite) with $\lambda_{j+1} \geq 2\lambda_j$ for all $j$. Then for
any $f:[0,N]^d \rightarrow [-1,1]$,
\[\sum_{j=1}^{J} |\mathcal E_{\lambda_j}(f)|^2 \leq C_{\epsilon}
N^{d}||f||_{4}^4 \leq C_{\epsilon} N^{2d},\]
where the constant $C_{\epsilon}$ depends only on the quantity $\epsilon$
used to define $\mathcal E_{\lambda}$ and, in particular, is
independent of $f$ and the
number $J$ of elements in the lacunary sequence.
\end{prop}
We provide details of this result in Section \ref{time-frequency section}.
\vskip0.2in
\begin{proof}[Proof of Theorem \ref{thm2.2}]
Assuming Propositions \ref{prop2.1}, \ref{prop2.2} and \ref{tf-prop}
for now,
the proof proceeds by contradiction. Assume that there exist
arbitrarily large $N$, a measurable set $A \subseteq [0,N]^d$
with $|A| \geq \delta N^d$, and a sequence of
non-admissible progression gaps $\lambda_1 <
\lambda_2 < \cdots < \lambda_J \ll N$ for some $J \geq J(\delta)$,
such that $\mathcal N_{\lambda_j}(f) = 0$ for $f= 1_A$. The sequence
may be chosen to be lacunary, and $J$ may be assumed to be arbitrarily
large as well, by choosing $N$ large enough. Thus, for $1 \leq j \leq J$,
\[ 0 = \mathcal N_{\lambda_j}(f) = c_1(\epsilon) \mathcal
M_{\lambda_j}(f) + \bigl[\mathcal N_{\lambda_j}(f) - \mathcal
M_{\lambda_j}^{\epsilon} (f)\bigr] + \mathcal E_{\lambda_j}(f).\]
In view of this, Propositions \ref{prop2.1}, \ref{prop2.2}, and recalling
that $c_1(\epsilon) \approx 1$, we find that for some sufficiently
small $\epsilon$ depending on $p, d$ and $\delta$, the inequality
\[ \frac{c(\delta)}{2} N^d \leq \bigl(c(\delta) - C_{p,d}
\epsilon^{\gamma p-1} \bigr)N^d \leq \bigl|c_1(\epsilon) \mathcal
M_{\lambda_j}(f) + \bigl[\mathcal N_{\lambda_j}(f) - \mathcal
M_{\lambda_j}^{\epsilon}(f) \bigr]\bigr|  = \bigl| \mathcal E_{\lambda_j}(f)\bigr|\]
holds for every $1 \leq j \leq J$.  Squaring both sides and summing
over all $j \leq J$ yields, after an application of Proposition
\ref{tf-prop} with $f = 1_{A}$,
\[ c(\delta) J N^{2d} \leq \sum_{j=1}^{J}|\mathcal
E_{\lambda_j}(f)|^2 \leq C_{\epsilon}N^{2d}. \]
This implies that $J \leq C_{p,d,\delta}$, contradicting the hypothesis that $J$
can be chosen arbitrarily large.
\end{proof}


\section{The main term} \label{main term section}

We now set about proving the main propositions leading up to the
theorem. In this section we prove Proposition \ref{prop2.1} via an application of Roth's theorem on compact abelian groups (see \cite{Bo}, Appendix A, Theorem 3). The compact group of interest is of course the $d$-dimensional torus $\Pi^d$.

\vskip0.2in
\begin{proof}[Proof of Proposition \ref{prop2.1}]
\vskip0.2in

Fix $\de\in(0,1]$ and $\la\ll N$ for a sufficiently large
$N$, and also let $A$ be given with the stated properties. Consider a real-valued function $f:[-N,N]^d\to [0,1]$ with $\int
f\geq \de N^d$.  The proposition, as phrased, is proven from what follows with the choice $f=\1_A$. \\

Equipartition the cube $[-N,N]^d$ in the natural way
into disjoint boxes with sides parallel to the coordinate axes and
length $\ell=c\la$ by choosing a sufficiently small number $c>0$ so
that $N/\ell$ is a positive integer. Enumerate the boxes by $\{B_i : 1
\leq i \leq L\}$. Each box $B_i$ may then be  identified with its
leftmost endpoint  $x_i$, so that $B_i=[0,\ell]^d +x_i$.
\vskip0.2in

Split $f$ into pieces restricted to each box. More precisely, define $g_i(x):[0,\ell]^d\to[0,1]$ by
$g_i(x)=\1_{[0,l]^n}(x+x_i)f(x+x_i)$, so that $f = \sum_{i=1}^{L}g_i$
on $[-N,N]^d$.  The non-negativity of $f$
implies the bound
\eq\label{3.1}
\MM_\la(f)\geq \sum_{i=1}^L\iint_{x,y\in\R^d}g_i(x)g_{i}(x+y)g_i(x+2y)\,
\om_{\lambda}(y)\,dy\,dx.\ee
Recall from \eqref{omega-scaling} that $\om_{\lambda}(y)=\la^{-d}\,\widehat{\psi}\,(\,|y/\la|^p-1\,)$;
hence we may choose $\psi$ such that $\widehat{\psi}\,(\,|y/\la|^p-1\,)\geq1/10$ for $y\in{[-\ell,\ell]^d}$. Then \eqref{3.1} yields
\begin{equation} \label{3.2}
\MM_\la(f)\geq \frac{\la^{-d}}{10}\sum_{i=1}^L\iint_{x,y\in\R^d} g_i(x)g_{i}(x+y)
g_i(x+2y)\,dy\,dx.\end{equation}
\vskip0.2in

Now identify $\Pi^d$ with the cube $[-\frac{1}{2},\frac{1}{2}]^d$. After a change of
variable $(x,y) \mapsto 10\ell (x,y)$, each summand on the right hand side of (\ref{3.2}) may be written as
\begin{align*}
(10\ell)^{2d}\iint_{x,y\in\R^d}
g_i(10\ell x)&g_{i}(10\ell(x+y))g_i(10\ell(x+2y))\,dy\,dx \\
=& (10\ell)^{2d}\iint_{x,y\in\Pi^d} g_i(10\ell x)g_{i}(10\ell(x+y))g_i(10\ell(x+2y))\,dy\,dx.
\end{align*}
Note that the support assumptions on the $g_i$ dictate that the integrand
is supported on $[-\frac{1}{10}, \frac{1}{10}]^{2d} \subset \Pi^{2d}$, as indicated in
the last step.
If we also have that
\begin{equation} \label{Roth-lemma-eq}\int_{\Pi^d} g_{i}(10\ell
  x)\,dx\geq \eta > 0
\end{equation}
for some index $i$ then Roth's theorem on compact abelian groups would
imply that for such an index
\[
\iint_{\Pi^d \times \Pi^d} g_i(10\ell x) g_{i}(10\ell (x+y)) g_i(10\ell (x+2y)) \,dx\,dy\, \geq \,c_0(d,\eta),
\]
where $c_0(d,\eta)>0$ is a constant uniform in $i$ but depending only on $d$ and $\eta$. We prove below in Lemma \ref{Roth-lemma}
that \eqref{Roth-lemma-eq} holds with $\eta = \delta
(10)^{-d}/2$ for at least $\delta L/2$ indices $i$. Summing over all these
indices in \eqref{3.2} then leads to the
bound
\[ \mathcal M_{\lambda}(f) \geq c_0(d, \eta) \lambda^{-d} \frac{\delta}{2}L
(10 \ell)^{2d} = c(d, \delta),\]
as claimed.
\end{proof}
\vskip0.2in
\begin{lemma} \label{Roth-lemma}
The relation \eqref{Roth-lemma-eq} holds with $\eta = \frac{\delta
  (10)^{-d}}{2}$ for at least $\delta L/2$ indices $i$.
\end{lemma}
\begin{proof}
This is a simple pigeonholing argument. Let $I$ denote the number of
indices $i$ for which the integral inequality in \eqref{Roth-lemma-eq}
holds. After a scaling change of variable, this is the same set of
indices $i$ for which $\int_{[0,\ell]^d} g_i(x) \, dx \geq (\de /2) \ell^d$. By our
hypothesis on $f$,
\begin{equation} \label{bound1}
\int_{\R^d} f(x) \, dx = \sum_{i=1}^L \int g_i(x) \, dx \geq \de N^d.
\end{equation}
On the other hand, $0 \leq g_i \leq 1$ is supported on $[0, \ell]^d$,
so $||g_i||_1 \leq \ell^d$ for trivial reasons. This leads to
the estimate
\begin{equation} \label{bound2}
\sum_{i=1}^L\int_{\R^d} g_i(x) dx\leq (\de/2)N^d + I \ell^d.\end{equation}
for each $1\leq i\leq L$. Combining the lower bound in \eqref{bound1}
with the upper bound in \eqref{bound2} and recalling that $L =
N^d/\ell^d$ leads to the claimed statement.
\end{proof}

\section{Error estimates} As indicated in the introduction, our main
objective here is to prove Proposition \ref{prop2.2}, which is a
direct consequence of Lemma \ref{lem2.1} and \ref{lem2.2}. Before
turning our attention to the proof of these lemmas, let us first make the simple but important
observation that $\int \om^\eps (y)\,dy \approx 1$ uniformly for
sufficiently small $\eps>0$. Let $\nu_p:=\bigl|\{y;\ |y|_p\leq
1\}\bigr|$, then by homogeneity $\bigl|\{y;\ |y|_p\leq
\eta\}\bigr|=\nu_p\,\eta^d$. This fact was used in the proof of
Theorem \ref{thm2.2} in order to bound from below the main term
$c_1(\epsilon) \mathcal M_{\lambda_j}(f)$.

\begin{lemma}\label{lemma-c1e} There exist constants $0<c_1<c_2$ depending only on the function $\psi$ and the parameter $d$ and $p$ such that
\eq\label{4.0.1}
c_1\leq \int_{\R^d} \om^\eps (y)\,dy \leq c_2,
\ee
uniformly for $0<\eps<\frac{1}{10d}$.
\end{lemma}

\begin{proof} By the definition of the function $\psi$ we have $|\widehat{\psi}(t)|\geq c$ for $|t|\leq \tau$ for some constants $c>0$ and $0<\tau\leq 1$. Then by \eqref{2.6} one has
\[
\int \om_1^\eps (y)\,dy \geq \frac{c}{\eps}\, \bigl|\{y;\ 1-\tau\eps\leq |y|^p\leq 1+\tau\eps\}\bigr|=\frac{c\nu_p}{\eps} \left((1+\tau\eps)^{d/p}-(1-\tau\eps)^{d/p}\right)\geq c_1,\]
uniformly for $0<\eps\leq 1/10d$ as $p>1$. Similarly, as $0\leq \widehat{\psi}\leq 1$ and $\widehat{\psi}$ is supported on $(-2,2)$,
\[\int \om^\eps (y)\,dy \leq \frac{1}{\eps}\, \bigl|\{y;\
1-2\eps\leq |y|^p\leq 1+2\eps\} \bigr|=\frac{\nu_p}{\eps} \left((1+2\eps)^{d/p}-(1-2\eps)^{d/p}\right)\leq c_2.\]
\end{proof}

Since $\om^\eps (y)$ is invariant under reflections to the coordinate hyperplanes we have that
\[\int \chi_+(y)\om^\eps(y)\,dy = 2^{-d} \int \om^\eps(y)\,dy,\]
and the above lemma holds for $\chi_+(y)\om^\eps(y)$ as well.
\vskip0.2in
To prove Lemma \ref{lem2.1} we need the following result.
\begin{lemma}\label{lem4.1} Let $f:[-N,N]^d\to [-1,1]$ and $g:[0,\la]^d\to [-1,1]$ be given functions. Then
\begin{equation} \label{fg-version}
\int_{x,y\in\R^d}f(x)f(x+y)f(x+2y)g(y)\,dx\,dy\,\ls\,N^{d}\,\la^{d/2}\,\|g\|_{U^3},
\end{equation}
where the implicit constant depends only on $d$.
\end{lemma}
\begin{proof}
The proof involves several changes of variables and successive applications of the Cauchy-Schwarz
inequality. Set
\[
T=\int_{x,y\in\R^d}f(x)f(x+y)f(x+2y)g(y)\,dx\,dy.\]
Applying the Cauchy-Schwarz inequality in the $x$ integration to get
\[
T^2\leq N^d\int _{x}\int_y\int_{y'}f(x+y)f(x+2y)f(x+y')f(x+2y')g(y)g(y')\,dx\,dy\,dy'
\]
Use the substitution $y'=y+h$ followed by the substitution $x\to x-y$, and define
\begin{equation} \Delta_hF(x)=F(x+h)\overline{F(x)} \label{Delta-def} \end{equation}
for a generic complex valued function $F$. Then one may write
\[
T^2\leq N^d\int _{x}\int_y\int_{h}\Delta_hf(x)\Delta_{2h}f(x+y)\Delta_hg(y)\,dx\,dy\,dh.
\]
The integrals in $y,\,h$ may be restricted to a region with $|y|,\,|h|\lesssim \la$ due to the support of $g$. Then another application of Cauchy-Schwarz in the $x$ and $h$ integration gives
\[
T^4\ls N^{3d}\la^d\int _{x,h,y,y'}\Delta_{2h}f(x+y)\Delta_{2h}f(x+y')\Delta_hg(y)\Delta_{h}g(y')\,dx\,dy\,dh\,dy'
\]
Again use the substitutions $y'=y+k$ and $x\to x-y$ in turn to get
\[
T^4\ls N^{3d}\la^d\int _{x,y,k,h}\Delta_{2h}f(x)\Delta_{2h}f(x+k)\Delta_hg(y)\Delta_{h}g(y+k)dx\,dy\,dh\,dk.
\]
One final application of the Cauchy-Schwarz inequality in $x$ and $h$ and $k$ integration gives
\[
T^8\ls N^{7d}\la^{4d}\int_{x,y,h,k,y'}1_{[0,N]^d}(x)\Delta_hg(y)\Delta_{h}g(y')\Delta_h g(y+k)\Delta_hg(y'+k)dx\,dy\,dy'\,dh\,\,dk.
\]
The $x$ integration may be carried out, and the applying the substitution $y'\to y+l$ gives the final form
\eq\label{normbound}
T^8 \ls N^{8d} \la^{4d}\int_{y,h,k,l} \Delta_{h,k,l}f(y)\,dy\,dh\,dk\,dl
\ee
where $\Delta_{h,k,l}$ is well defined as the composition of the operators $\Delta_h$, $\Delta_k$, and $\Delta_l$. The integral is easily verified to be $\|g\|_{U^3}^8$, which completes the proof.
\end{proof}
\vskip0.2in
\begin{proof}[Proof of Lemma \ref{lem2.1}]
As indicated in the introduction, the right hand side of \eqref{2.11} is
to be
interpreted as \[ ||(\sigma - \omega^{\epsilon})||_{U^3} :=
\lim_{\eta \rightarrow 0}  ||(\omega^{\eta} -
\omega^{\epsilon})||_{U^3}. \] Since the integral representation of \[ \mathcal N_{\lambda}(f) - \mathcal M_{\lambda}^{\epsilon}(f)
= \lim_{\eta \rightarrow 0} \mathcal M_{\lambda}^{\eta}(f) - \mathcal
M_{\lambda}^{\epsilon}(f)  \]
is of the form \eqref{fg-version}, we may apply Lemma \ref{lem4.1}
with $g(y) =  \bigl(\omega_{\lambda}^{\eta}(y) -
\omega_{\lambda}^{\epsilon}(y) \bigr)$ to get
\[ \bigl| \mathcal N_{\lambda}(f) - \mathcal
M_{\lambda}^{\epsilon}(f) \bigr| \lesssim N^d \lambda^{d/2} \lim_{\eta
\rightarrow 0} ||\chi_{+}  \bigl(\omega_{\lambda}^{\eta} -
\omega_{\lambda}^{\epsilon}\bigr) ||_{U^{3}}.\]
Since $\omega_{\lambda}^{\epsilon}$ is a rescaled version of
$\omega^{\epsilon}:=\omega_1^{\epsilon}$, scaling properties of the $U^3$ norm imply that \[ ||(\omega_{\lambda}^{\eta} -
\omega_{\lambda}^{\epsilon}\bigr) ||_{U^{3}} = \lambda^{-d} \lambda^{d/2}
|| \bigl( \omega^{\eta} - \omega^{\epsilon}\bigr)||_{U^3},  \]
which leads to the claimed upper bound.
\end{proof}
\vskip0.2in
Next we turn to the proof of Lemma \ref{lem2.2}. In what follows we
assume that $\la$ and $N\gg\la$ are fixed, and $f$ is the
characteristic function of a set $A\subset[-N,N]^d$ with measure $\de
N^d$. First we need an estimate for one-dimensional scalar oscillatory integrals of the following type.

\begin{lemma}\label{lem4.2} Let $1 < p < \infty$, $p\neq 2$. For a
  smooth cut-off function $\phi$ on $\mathbb R$, let $\phi_+$ be its restriction to the positive real numbers and define the integral
\eq\label{4.2.1}
I(t):=\int_{y,h,k,l \in \mathbb R} \tr_{h,k,l} \left(\phi_+(y)e^{it|y|^p}\right)\,dy\,dh\,dk\,dl,
\ee
where $\Delta_{h,k,l}$ is defined via iterated compositions of
$\Delta_h$ as described in \eqref{Delta-def} and the proof of Lemma
\ref{lem4.1}.
Then there exists a constant $r=r(p)>0$ such that
\eq\label{4.2.2}
|I(t)|\leq C_p |t|^{-\frac{1}{r}},\quad \textit{for}\ |t|\geq 1.
\ee
One may take $r(p)=p+1$ for $1<p<2$, and $r(p)=2p-1$ for $p>2$. The
constant $C_p$ is finite in the indicated range of $p$, and tends to
infinity as $p \rightarrow 1$ or $2$.
\end{lemma}

\begin{proof} Replacing $y+h$ by a new variable $y'$, the integral $I(t)$ may be rewritten as
\[I(t)=\int |I_{k,l}(t)|^2\,dk\,dl,\]
where
\[I_{k,l}(t)=\int \tr_{k,l}\phi_+(y)\,e^{it\psi_{k,l}(y)}\,dy,\]
with
\[\tr_{k,l}\phi_+(y)=\phi_+(y)\phi_+(y+k)\phi_+(y+l)\phi_+(y+k+l),\]
\[\psi_{k,l}(y)=y^p+(y+k+l)^p-(y+k)^p-(y+l)^p.\]
The reason we can write $\psi_{k,l}(y)$ in this form is that $y,\,y+k,\,y+l,\,y+k+l$ are all positive on the support of $\tr_{k,l}\phi_+$.
\vskip0.2in
It is clear that $I_{k,l}(\eta)$ is uniformly bounded, hence $I(t)$
receives small contribution from regions where at least one of the
integration variables $k,l$ is small. For a small
parameter $0<\eta<1$ to be chosen later, we may therefore write
\[I(t)=\int_{|k|,|l|\geq\eta} |I_{k,l}(t)|^2\,dk\,dl + O(\eta).\] We now estimate the integral $I_{k,l}(t)$ for fixed $k,l$ assuming $|k|,|l|\geq\eta$. Introducing a smooth partition of unity, we have
$\,I_{k,l}(t)=I'_{k,l}(t)+J_{k,l}(t),\,$ where the domain of integration of $I'_{k,l}(t)$ ranges over those $y$ for which at least one of the quantities $y,\ y+k,\ y+l,\,y+k+l$ is $O(\eta)$. Thus $I'_{k,l}(t)=O(\eta)$.
\vskip0.2in
For $J_{k,l}(t)$ one may write, using Taylor's remainder formula
\[
\psi_{k,l}(y)= klp(p-1) \int_{[0,1]^2} (y+uk+sl)^{p-2}\,du\,ds.
\]
Therefore its derivative is given by
\[
\psi'_{k,l}(y)= klp(p-1)(p-2) \int_{[0,1]^2} (y+uk+sl)^{p-3}\,du\,ds.
\]
By our assumptions, we have that
\[\eta\ls y+uk+sl\ls 1,\] uniformly for $0\leq u,s\leq 1$, and also that $|k|,|l|\gs \eta$. Thus for $p>1$, $p\neq 2$,
\[|\psi_{k,l}'(y)|\gs \eta^{p-1},\]
with an implicit constant independent of $k$ and $l$. Then, writing $\psi=\psi_{k,l}$ and $\chi$ for the amplitude, integration by parts yields
\begin{align*}
J_{k,l}(t)&= \int_{\R^d}\frac{d}{dy} \left(e^{it\psi(y)}\right)\ \frac{\chi(y)}{it\psi'(y)}\,dy
=-\frac{1}{it} \int_{\R^d} e^{it\psi(y)}\ \frac{d}{dy} \left(\frac{\chi(y)}{\psi'(y)}\right)\,dy\\
&= \frac{1}{it} \int_{\R^d} e^{it\psi(y)}\ \left( \frac{\chi'}{\psi'} + \frac{\chi\psi''}{\psi'^2}\right)\,dy.
\end{align*}
Here we have used the support properties of $\chi$ in the form  $||\chi||_{\infty}=O(1)$ and
$||\chi'||_{\infty} = O(\eta^{-1})$. Therefore
\[|J_{k,l}(t)|\ls |t|^{-1}\,(\eta^{-p}+\eta^{-2p+2})\leq |t|^{-1} \eta^{-r'(p)},\]
with $r'(p)=\max (p,2p-2)>0$. This implies that $I_{k,l}(t)=O(\eta)+O(|t|^{-1}\eta^{-r'_p})$, choosing $\eta:= |t|^{-\frac{1}{r_p+1}}$ yields
\eq\label{4.2.3}
|I(t)|\ls |t|^{-\frac{1}{r_p}},
\ee
with $r(p)=r'(p)+1=\max (p+1,2p-1)$ and the lemma follows.
\end{proof}

\bigskip

\begin{proof}[Proof of Lemma \ref{2.2}]
By the triangle inequality \[ ||\sigma -
\omega^{\epsilon}||_{U^3(\mathbb R^d)} \leq \sum_{i=1}^{2^d} ||\chi_i(\sigma -
\omega^{\epsilon})||_{U^3(\mathbb R^d)},\] where $\chi_i$ is the
indicator function of the $i$th orthant of $\mathbb
R^d$. Since both $\sigma$ and $\omega^\epsilon$ are invariant under
reflections about the coordinate hyperplanes, it suffices to estimate
$||\chi_{+}(\sigma - \omega^{\epsilon})||_{U^3}$, where $\chi_{+}$ is
the indicator function of the positive orthant. Recalling that $\omega^{\eta} - \omega^{\epsilon}$ is compactly supported,
say on $[-C, C]^d$, let us fix a smooth cutoff function
$\Phi(y)=\phi^{\otimes d}(y)$ where $\phi$ is a smooth bump function
supported on an interval of the form $[-2C,2C]$ and identically 1 on
the middle half of it. Then,
\begin{equation} \label{eqss}
\|\chi_+(\om^{\eta}-\om^\eps)\|_{U^3} = \|\Phi_+(y)\int_t(\psi(\eta
t))-\psi(\eps t))e^{i(|y|^p-1)t}dt||_{U^3(dy)},
\end{equation}
with $\Phi_+(y)=\chi_+(y)\Phi(y)$.
Applying Minkowski's inequality to the right hand side of the last equation \eqref{eqss} this is further estimated by
\begin{equation} \label{total integral}
\int_t|\psi(\eta  t))-\psi(\eps  t)|\,||\Phi_+(y)e^{it\,|y|^p}||_{U^3(\R^d)}dt.
\end{equation}
Note that as $\Phi_+(y)e^{it\,|y|^p}=\prod_{i=1}^d \phi_+(y_i)e^{it\,y_i^p}$, we have
\[
||\Phi_+(y)e^{it\,|y|^p}||_{U^3(\mathbb R^d)}\,=\, ||\phi_+(y)e^{it\,y^p}||_{U^3(\R)}^d,
\]
where the one-dimensional integrals
\[
||\phi_+(y)e^{it\,y^p}||_{U^3(\R)}^8=\int_{y,h,k,l\in\R}\left(\Delta_{h,k,l}\phi_+(y)e^{it|y|^p}\right)
\,dy\,dh\,dk\,dl
\]
are estimated in Lemma \ref{lem4.2}. Thus, we have for $|t|\geq 1$
\[
||\Phi_+(y)e^{it\,|y|^p}||_{U^3(y)}\,\ls\,|t|^{-\frac{d}{8r}},
\]
with $r=r(p)$ given in \eqref{4.2.2}. Inserting this bound into
\eqref{total integral}, we complete the estimation as follows,
\begin{align*}
\int_t|\psi(\eta t))-\psi(\eps t)|\,|t|^{-\frac{d}{8r}}\,dt & \leq
\int_t \bigl[|\psi(\eta t)| + | \psi(\epsilon t)|\bigr] t^{-d/8r}\,dt \\
& \lesssim\,\eta^{\frac{d}{8r}-1}  + \eps^{\frac{d}{8r}\,-1} \lesssim \epsilon^{\frac{d}{8r}-1}
\end{align*}
which is bounded uniformly in $\eta$ provided that $\eta \ll \epsilon$
and $d>8r$.
\end{proof}


\section{A result from time-frequency analysis} \label{time-frequency
  section} Here we will prove Proposition \ref{tf-prop} by using the main
result of \cite{MTT}. The necessary verifications of the hypotheses of
\cite{MTT} will be done subsequently.
\begin{proof}[Proof of Proposition \ref{tf-prop}]
By the Cauchy-Schwarz inequality and support restrictions on $f$,
\begin{align}\label{2.15}
|\mathcal{E}_{\la_j}(f)|^2\,&\ls\,N^d\,\iiint_{(\R^d)^3} f(x+y)f(x+z)f(x+2y)f(x+2z)\,k_j^\eps(y)k_j^\eps(z)\,dy\,dz\,dx\nonumber\\
&=\,N^d\,\iiint_{(\R^d)^3} f(x)f(x+z-y)f(x+y)f(x+2z-y)\,k_j^\eps(y)k_j^\eps(z)\,dy\,dz\,dx.
\end{align}
Summing \eqref{2.15} for $1\leq j\leq J$ one has
\eq\label{2.16}
\sum_{j=1}^J |\mathcal E_{\la_j}(f)|^2\,\ls\,\,N^d\,\int_{\R^d} f(x)\, \mathcal{K}^\eps_J(f,f,f)(x)\,dx,
\ee
where the trilinear operator $\mathcal K_J^{\epsilon}$ on the right side is given by
\eq\label{2.17}
\mathcal{K}_J^\eps (f_1,f_2,f_3)(x):=\int\int_{(\R^d)^2} f_1(x+z-y)f_2(x+y)f_3(x+2z-y)K_J^\eps(y,z)\,dy\,dz,
\ee
with integration kernel
\begin{equation}\label{2.18}
K_J^\eps(y,z)=\sum_{j=1}^J k_j^\eps (y)k_j^\eps (z), \quad \text{
  where } \quad k_j^{\epsilon}(y) = \omega_{\lambda_j}(y) -
c_1(\epsilon) \omega_{\lambda_j}(y).
\end{equation}
We justify in Lemma \ref{lem5.1} below that this is a
Calder\'on-Zygmund kernel.  Lebesgue mapping properties of associated
multilinear operators of the type given in \eqref{2.17} were studied by Muscalu, Tao and Thiele \cite{MTT}. Indeed, by Theorem 1.1 in \cite{MTT} one has the estimate
\begin{equation}\label{2.19}
\|\mathcal{K}_J^\eps (f,f,f)\|_{L^{4/3}(\R^d)}\,\leq\,C_\eps\,\|f\|^3_{L^4(\R^d)},
\end{equation}
with a constant $C_\eps >0$ independent of $J$ and the sequence
$\la_1<\ldots <\la_J$. Then
\begin{align}\label{2.20}
\sum_{j=1}^J |\mathcal E_{\la_j}(f)|^2\,&\ls\,\,N^d\,\int_{\R^d} f(x) \mathcal{K}^\eps_J(f,f,f)(x)\,dx\nonumber\\
&\leq\,C_\eps\,N^d\, \|\mathcal K_J^{\epsilon}(f, f,
f)\|_{L^{4/3}(\R^d)}\,\|f\|_{L^4(\R^d)} \leq C_{\epsilon} N^d ||f||_4^4,
\end{align}
as claimed.
\end{proof}

In preparation for Lemma \ref{lem5.1}, let us begin by rewriting \eqref{2.20} as
\eq\label{5.1}
\mathcal{K}_J^\eps (f_1,f_2,f_3)(x):=\iiint_{(\R^d)^3} e^{ix\cdot(\xi_1+\xi_2+\xi_3)}\widehat{f_1}(\xi_1)\widehat{f_2}(\xi_2)
\widehat{f_3}(\xi_3)m(\xi_1,\xi_2,\xi_3)\,d\xi_1\,d\xi_2\,d\xi_3,
\ee
where
\begin{align}
m(\xi_1,\xi_2,\xi_3) &=\iint_{(\R^d)^2} K_J^\eps
(y,z)e^{-iy\cdot(\xi_1-\xi_2+\xi_3)}e^{ix\cdot(\xi_1+2\xi_3)}dy\,dz
\nonumber \\
&=
\widehat{K_J^\eps}(-\xi_1+\xi_2-\xi_3,\xi_1+2\xi_3) \label{5.2}\end{align}
Set \begin{align}
\Gamma &=\{(\xi_1,\xi_2,\xi_3,\xi_4) \in (\mathbb R^d)^4 :\xi_1+\xi_2+\xi_3+\xi_4=0\},
\text{ and } \label{def-Gamma} \\
\Gamma' &= \{ (\xi_1,\xi_2,\xi_3,\xi_4) \subset\Gamma :
\xi_1-\xi_2+\xi_3=0, \; \xi_1+2\xi_3=0, \label{def-Gamma'}
\end{align}
so that $\text{dim}(\Gamma)=3d$ and $\text{dim}(\Gamma')=d$. By taking the Fourier
transform one may re-express $\mathcal K_J^{\epsilon}$ as a
multiplier, namely,
\eq\label{5.3}
\widehat{\mathcal{K}}_J^\eps (f_1,f_2,f_3)(-\xi_4):=\int_\Gamma \widehat{f_1}(\xi_1)\widehat{f_2}(\xi_2)
\widehat{f_3}(\xi_3)m(\xi_1,\xi_2,\xi_3,\xi_4)\,d\xi_1\,d\xi_2\,d\xi_3,
\ee
identifying the operator $\mathcal{K}_J^\eps$ with the ones studied in
\cite{MTT}. The only difference is that the functions $f_i$ are now
defined on $\R^d$ instead of on $\R$; however this does not affect the
arguments given there. To clarify, $\Gamma'$ is a graph over $d$
of the canonical variables, and $\mathcal K_J^{\epsilon}$ generates a
4-linear form, so in the notation of \cite{DPT}, $k=d$ and $n =
4$. Thus the rank of the operator as described in this paper is $m =
k/d = 1$, an integer $< 2 = n/2$. As has been pointed out in page 295
of \cite{DPT}, the higher-dimensional adaptation of the main result of
\cite{MTT} in this setting of integral rank is fairly straightforward,
provided certain requirements on the multiplier is met. Thus, in order
to apply the main result of \cite{MTT} one needs to establish certain
growth and differentiability properties of the multiplier
$m(\xi)$. This is the goal of the following lemma.

\begin{lemma}\label{lem5.1} Given an integration kernel $K_{J}$ as in \eqref{2.18}, with the
  summands obeying the cancellation condition \eqref{2.14.5}, let $m(\xi)$ be the
  associated multiplier defined in \eqref{5.2}, where
  $\xi=(\xi_1,\ldots,\xi_3) \in \mathbb R^{3d}$ is a coordinate system
  identifying the subspace $\Gamma$. Then for any multi-index $\al
  \in \mathbb Z_{+}^{3d}$ one has the estimate
\eq\label{5.4}
|\partial_{\xi}^\al m(\xi)|\,\leq C_{\al,\eps}\, (dist(\xi,\Gamma'))^{-|\al|}.
\ee
Here $\Gamma$ and $\Gamma'$ are as in \eqref{def-Gamma} and \eqref{def-Gamma'}.
\end{lemma}

{\em{Remark: }} The crucial point here is that the constant
$C_{\al,\eps}$ is independent of $J$ and the lacunary sequence
$\la_1<\ldots <\la_J$. Also the above estimate is needed just up to
some fixed finite order. Once this is established our main result
Theorem \ref{thm2.2} follows as explained at the end of Section 2.

\begin{proof} Since $m$ is essentially $\widehat{K_J^{\epsilon}}$
  composed with a linear transformation, we study the latter function
  in detail. The relation \eqref{2.18} implies that
\[
\widehat{K_J^\eps}(\eta,\zeta)=\sum_{j=1}^J\widehat{k^\eps_j}(\eta)
\widehat{k^\eps_j}(\zeta)
\]
and
\[
\widehat{k^\eps_j}(\eta)=\int_{y\in\R^d} e^{iy\cdot\eta}
\left(\om^\eps_{\la_j}(y)-c_1(\eps)\om_{\la_j}(y)\right)\,dy=
\widehat{\om^\eps} (\la_j\eta)-c_1(\eps)\widehat{\om}(\la_j\eta).
\]
Let us recall the definitions
$\om_1^\eps(y)=\eps^{-1}\widehat{\psi}((|y|^p-1)/\eps)$, where $\om(y)=\om_1^1(y)$ with $\widehat{\psi}$ is a compactly supported smooth function. Therefore for all multi-indices $\al$
\eq\label{5.5}
|\partial_\eta^\al \widehat{\om_1^\eps}(\eta)|\,\leq C_\al.
\ee
Integrating by parts $k$ times in the integral expression for
$\partial_{\eta}^{\alpha} \omega_1^{\epsilon}$ we also obtain
\eq\label{5.6}
|\partial_\eta^\al\, \widehat{\om_1^\eps}(\eta)|\,\leq C_{\al,k}\,\eps^{-|\al|}\,(1+|\eta|)^{-k}.
\ee
Thus for any $k\in\N$
\eq\label{5.7}
|\partial_\eta^\al\,\widehat{k^\eps_j}(\eta)|\,\leq\,
C_{\al,k}\,\eps^{-|\al|}\,(1+|\la_j\eta|)^{-k}.
\ee
The cancellation property \eqref{2.14.5} gives that
$\widehat{k^\eps_j}(0)=0$, which leads to the additional pointwise estimate
\eq\label{5.8}
|\widehat{k^\eps_j}(\eta)|\,\leq\,C\,\la_j |\eta|.
\ee
This implies that $\widehat{K_J^\eps}(0,\zeta)=\widehat{K_J^\eps}
(\eta,0)=0$ and
\eq\label{5.9}
|\widehat{K_J^\eps}(\eta,\zeta)|\,\lesssim\,\sum_{j\leq J} \min\, \left(\la_j|\eta|,\frac{1}{\la_j|\eta|}\right)\,\lesssim\,1
\ee
as the sequence $\mu_j:=\la_j|\eta|$ is lacunary and  $\|\widehat{k^\eps_j}\|_\infty \lesssim 1$. This shows that the multiplier $m(\xi)$ defined in \eqref{5.3} is bounded. To estimate its partial derivatives we apply \eqref{5.7} with $k=|\al|+|\be|+1$ to write
\begin{align*}
|\partial_\eta^\al \partial_\zeta^\be\ \widehat{K_J^\eps}(\eta,\zeta)|\,&\leq\, C_{\al,\be,\eps}\,\sum_{j\leq J}\,\la_j^{|\al|+|\be|}\, \min\left( (1+|\la_j\eta|)^{-|\al|-|\be|-1},(1+|\la_j\zeta|)^{-|\al|-|\be|-1}\right)\\
&\leq\,C_{\al,\be,\eps}\,\min\,(|\eta|^{-|\al|-|\be|},|\zeta|^{-|\al|-|\be|})\,
\leq\,C_{\al,\be,\eps}\,(|\eta|+|\zeta|)^{-|\al|-|\be|}.\nonumber
\end{align*}
Here we have used the fact that
\[\sum_{j\leq J} \mu_j^k (1+\mu_j)^{-k-1}\leq \sum_{j:\,\mu_j\leq 1}\mu_j +\sum_{j:\,\mu_j\geq 1} \mu_j^{-1}\,\leq\,C,\]
for the lacunary sequences $\mu_j:=\la_j|\eta|$ and
$\mu_j:=\la_j|\zeta|$ with $k=|\al|+|\be|\geq 1$. By \eqref{5.2} this leads to the estimate
\[|\partial_\xi^\al\, m(\xi)|\,\leq\,C_{\al,\eps}\,(|-\xi_1+\xi_2-\xi_3|^{-|\al|}+|\xi_1+2\xi_3|^{-|\al|})
\,\leq\,C'_{\al,\eps}\,dist(\xi,\Gamma')^{-|\al|}.
\]
for $\xi=(\xi_1,\xi_2,\xi_3)\in\R^{3d}\simeq\Gamma$ and any
multi-index $\al\in\Z_+^{3d}$. This proves Lemma \ref{lem5.1}.
\end{proof}

\vskip0.2in
\noindent \author{\textsc{Brian Cook}}\\
University of Wisconsin, Madison, USA. \\
Electronic address: \texttt{bcook@math.wisc.edu}
\vskip0.2in
\noindent \author{\textsc{Akos Magyar}}\\
University of Athens, Georgia, USA. \\
Electronic address: \texttt{magyar@math.uga.edu}
\vskip0.2in
\noindent \author{\textsc{Malabika Pramanik}}\\
University of British Columbia, Vancouver, Canada. \\
Electronic address: \texttt{malabika@math.ubc.ca}


\vspace{.5in}


\begin{thebibliography}{18}

\bibitem{Bo}{\sc J. Bourgain}, {\em A Szemer\'edi type theorem for sets of positive density in $\R^k$}. Israel J. Math., 54(3), (1986), 307-316.

\bibitem{CM}
{\sc B. Cook, \'A Magyar}, {\em On restricted arithmetic progressions over finite fields},
Online J. Anal. Comb. 7 (2012), 1-10

\bibitem{DPT}
{\sc C. Demeter, M. Pramanik, C. Thiele}, {\em Multilinear singular
  operators with fractional rank},
 Pacific J. Math. 246 (2010), no. 2, 293–324.

\bibitem{EGMRSS}
{\sc P. Erd\"os, R.L. Graham, P. Montgomery, B.L. Rothschild,
  J.H. Spencer, E.G. Straus}, {\em{Euclidean Ramsey theorems}},
J. Combin. Theory Ser. A 14 (1973), 341-363.

\bibitem{FM}
{\sc K. Falconer, J. Marstrand}, {\em Plane sets with positive density at infinity contain all large distances},
 Bull. London Math. Soc. 18 (1986), no. 5, 471–474.


\bibitem{FKW}
{\sc H. Furstenberg, Y. Katznelson, B. Weiss}, {\em Ergodic theory and configurations in sets of positive density}, Mathematics of Ramsey theory. Springer Berlin Heidelberg, (1990), 184-198.

\bibitem{Gow}
{\sc W.T. Gowers}, {\em A new proof of Szemer\'edi's theorem for arithmetic progressions of length four}, Geometric and Functional Analysis, 8(3), (1998) 529-551.

\bibitem{Gr}
{\sc R.L. Graham}, {\em Recent trends in Euclidean Ramsey theory},
Discrete Mathematics 136 (1994), 119-127.

\bibitem{GRS}
{\sc R.L. Graham, B.L. Rothschild, J.H. Spencer}, {\em Ramsey Theory},
John Wiley, New York, second edition 1990.

\bibitem{Ho}
{\sc L. H\"ormander} {\em The Analysis of Linear Partial Differential Operators. Vol. I: Distribution Theory and Fourier Analysis}, Springer Berlin-Heidelberg (1983)

\bibitem{MTT} {\sc C. Muscalu, T. Tao, C. Thiele}, {\em  Multi-linear operators given by singular multipliers},  Journal of the AMS, 15.2 (2002), 469-496.

\bibitem{Ro} {\sc K. F. Roth}, {\em On certain sets of integers},  J. London Math. Soc. 28, (1953). 104–109.

\bibitem{Sz} {\sc E. Szemer\'edi}, {\em On sets of integers containing no
  $k$ elements in arithmetic progression}, Acta Arith. 27 (1975),
199-245.

\bibitem{vdW} {\sc B.L. van der Waerden}, {\em Beweis einer
    Baudetschen Vermutung}, Nieuw Arch. Wisk 15 (1927), 212-216.

\end{thebibliography}
\end{document}